\documentclass[reqno,oneside]{amsart}

\usepackage[danish,english]{babel} 
\usepackage[latin1]{inputenc} 
\usepackage[T1]{fontenc} 
\usepackage{lmodern}
\usepackage{amsmath,amsthm,amssymb,amsfonts,mathrsfs,latexsym} 
\usepackage{graphicx} 
\usepackage{verbatim} 
\usepackage[all]{xy} 
\usepackage[hyperfootnotes=false, plainpages=false, pdfpagelabels, colorlinks, linkcolor=red, citecolor=blue, urlcolor=blue, hypertexnames=true]{hyperref} 
\usepackage{multirow}

\makeatletter
\newtheorem*{rep@theorem}{\rep@title}
\newcommand{\newreptheorem}[2]{%
\newenvironment{rep#1}[1]{%
 \def\rep@title{#2 \ref{##1}}%
 \begin{rep@theorem}}%
 {\end{rep@theorem}}}
\makeatother

\newtheorem{thm}{Theorem}[section]

\newtheorem{prop}[thm]{Proposition}
\newtheorem{cor}[thm]{Corollary}
\newtheorem*{thm*}{Theorem}
\newtheorem*{problem*}{Problem}
\newtheorem{question}[thm]{Question}

\newtheorem*{claim*}{Claim}
\theoremstyle{definition}
\newtheorem{defi}[thm]{Definition}

\newtheorem{rem}[thm]{Remark}

\newcommand{\mc}{\mathcal}

\newcommand{\mf}{\mathfrak}
\newcommand{\mr}{\mathrm}

\newcommand{\C}{\mathbb{C}}
\newcommand{\N}{\mathbb{N}}
\newcommand{\Q}{\mathbb{Q}}
\newcommand{\R}{\mathbb{R}}
\newcommand{\Z}{\mathbb{Z}}

\newcommand{\la}{\langle}
\newcommand{\ra}{\rangle}
\renewcommand{\epsilon}{\varepsilon}
\renewcommand{\phi}{\varphi}
\renewcommand{\tilde}{\widetilde}
\renewcommand{\hat}{\widehat}
\newcommand{\acts}{\curvearrowright}
\newcommand{\dist}{\mr{dist}}

\DeclareMathOperator{\GL}{GL}
\DeclareMathOperator{\SL}{SL}
\DeclareMathOperator{\SO}{SO}
\DeclareMathOperator{\SU}{SU}
\DeclareMathOperator{\Sp}{Sp}
\DeclareMathOperator{\PSL}{PSL}

\DeclareMathOperator{\Spin}{Spin}

\DeclareMathOperator{\supp}{supp}

\newcommand{\dd}{\mr d}

\newcommand{\FF}{\mr F_{4(-20)}}
\newcommand{\WA}{\mr{WA}}
\newcommand{\WH}{\mr{WH}}

\DeclareFontFamily{U}{mathx}{\hyphenchar\font45}
\DeclareFontShape{U}{mathx}{m}{n}{
      <5> <6> <7> <8> <9> <10>
      <10.95> <12> <14.4> <17.28> <20.74> <24.88>
      mathx10
      }{}
\DeclareSymbolFont{mathx}{U}{mathx}{m}{n}
\DeclareFontSubstitution{U}{mathx}{m}{n}
\DeclareMathAccent{\widecheck}{0}{mathx}{"71}
\DeclareMathAccent{\wideparen}{0}{mathx}{"75}

\numberwithin{equation}{section}

\begin{document}
\selectlanguage{english} 

\begin{abstract}
In this paper we consider the class of connected simple Lie groups equipped with the discrete topology. We show that within this class of groups the following approximation properties are equivalent: (1) the Haagerup property; (2) weak amenability; (3) the weak Haagerup property (Theorem~\ref{thm:main}). In order to obtain the above result we prove that the discrete group $\GL(2,K)$ is weakly amenable with constant 1 for any field $K$ (Theorem~\ref{thm:WA1}).
\end{abstract}


\title[Approximation properties of simple Lie groups made discrete]{\texorpdfstring{Approximation properties of \\simple Lie groups made discrete}{Approximation properties of simple Lie groups made discrete}}

\author{S{\o}ren Knudby}
\address{Department of Mathematical Sciences, University of Copenhagen,
\newline Universitetsparken 5, DK-2100 Copenhagen \O, Denmark}
\email{knudby@math.ku.dk}

\author{Kang Li}
\address{Department of Mathematical Sciences, University of Copenhagen,
\newline Universitetsparken 5, DK-2100 Copenhagen \O, Denmark}
\email{kang.li@math.ku.dk}

\thanks{Both authors are supported by ERC Advanced Grant no.~OAFPG 247321 and the Danish National Research Foundation through the Centre for Symmetry and Deformation (DNRF92).}

\date{January 28, 2015}
\maketitle
\parindent 0cm
\parskip 4pt

\section{Introduction}
Amenability for groups was first introduced by von Neumann in order to study the Banach-Tarski paradox. It is remarkable that this notion has numerous characterizations and one of them, in terms of an approximation property by positive definite functions, is the following: a locally compact (Hausdorff) group $G$ is amenable if there exists a net of continuous compactly supported, positive definite functions on $G$ tending to the constant function 1 uniformly on compact subsets of $G$.
Later, three weak forms of amenability were introduced: the Haagerup property, weak amenability and the weak Haagerup property. In this paper we will study these approximation properties of groups within the framework of Lie theory and coarse geometry.
\begin{defi}[Haagerup property \cite{MR1852148}]
A locally compact group $G$ has the \emph{Haagerup property} if there exists a net of positive definite $C_0$-functions on $G$, converging uniformly to $1$ on compact sets.
\end{defi}

\begin{defi}[Weak amenability \cite{MR996553}]
A locally compact group $G$ is \emph{weakly amenable} if there exists a net $(\phi_i)_{i\in I}$ of continuous, compactly supported Herz-Schur multipliers on $G$, converging uniformly to $1$ on compact sets, and such that $\sup_i \|\phi_i\|_{B_2} < \infty$.

The \emph{weak amenability constant} $\Lambda_\WA(G)$ is defined as the best (lowest) possible constant $\Lambda$ such that $\sup_i \|\phi_i\|_{B_2} \leq\Lambda$, where $(\phi_i)_{i\in I}$ is as just described.
\end{defi}

\begin{defi}[The weak Haagerup property \cite{K-WH}]
A locally compact group $G$ has the \emph{weak Haagerup property} if there exists a net $(\phi_i)_{i\in I}$ of $C_0$ Herz-Schur multipliers on $G$, converging uniformly to $1$ on compact sets, and such that $\sup_i \|\phi_i\|_{B_2} < \infty$.

The \emph{weak Haagerup constant} $\Lambda_\WH(G)$ is defined as the best (lowest) possible constant $\Lambda$ such that $\sup_i \|\phi_i\|_{B_2} \leq\Lambda$, where $(\phi_i)_{i\in I}$ is as just described.
\end{defi}
Clearly, amenable groups have the Haagerup property. It is also easy to see that amenable groups are weakly amenable with $\Lambda_\WA(G)=1$ and that groups with the Haagerup property have the weak Haagerup property with $\Lambda_\WH(G)=1$. Also, $1\leq \Lambda_\WH(G)\leq\Lambda_\WA(G)$ for any locally compact group $G$, so weakly amenable groups have the weak Haagerup property.

It is natural to ask about the relation between the Haagerup property and weak amenability. The two notions agree in many cases, like generalized Baumslag-Solitar groups (see \cite[Theorem~1.6]{cornulier-valette-BS}) and connected simple Lie groups with the discrete topology (see Theorem~\ref{thm:main}). However, in the known cases where the Haagerup property coincides with weak amenability, this follows from classification results on the Haagerup property and weak amenability and not from a direct connection between the two concepts.
In general, weak amenability does not imply the Haagerup property and vice versa. In one direction, the group $\Z/2\wr \mathbb{F}_2$ has the Haagerup property \cite{MR2393636}, but is not weakly amenable \cite{MR2914879}. In the other direction, the simple Lie groups $\Sp(1,n)$, $n\geq 2$, are weakly amenable \cite{MR996553}, but since these non-compact groups also have Property (T) \cite[Section~3.3]{MR2415834}, they cannot have the Haagerup property. However, since the weak amenability constant of $\Sp(1,n)$ is $2n-1$, it is still reasonable to ask whether $\Lambda_\WA(G) = 1$ implies that $G$ has the Haagerup property. In order to study this, the weak Haagerup property was introduced in \cite{MR3146826,K-WH}, and the following questions were considered.
\begin{question}\label{question:constants}
For which locally compact groups $G$ do we have $\Lambda_\WA(G)=\Lambda_\WH(G)$?
\end{question}
\begin{question}\label{question:H-WH}
Is $\Lambda_\WH(G) = 1$ if and only if $G$ has the Haagerup property?
\end{question}
It is clear that if the weak amenability constant of a group $G$ is 1, then so is the weak Haagerup constant, and Question~\ref{question:constants} has a positive answer. In general, the constants differ by the example $\Z/2\wr \mathbb{F}_2$ mentioned before. There is an another class of groups for which the two constants are known to be the same.
\begin{thm}[\cite{HK-WH-examples}]
Let $G$ be a connected simple Lie group. Then $G$ is weakly amenable if and only if $G$ has the weak Haagerup property. Moreover, $\Lambda_\WA(G) = \Lambda_\WH(G)$.
\end{thm}
By the work of many authors
\cite{MR748862,MR996553,MR784292,MR1418350,UH-preprint,MR1079871}, 
 it is known that a connected simple Lie group $G$ is weakly amenable if and only if the real rank of $G$ is zero or one. Also, the weak amenability constants of these groups are known. Recently, a similar result was proved about the weak Haagerup property \cite[Theorem~B]{HK-WH-examples}. Combining the results on weak amenability and the weak Haagerup property with the classification of connected Lie groups with the Haagerup property \cite[Theorem~4.0.1]{MR1852148} one obtains the following theorem, which gives a partial answer to both Question~\ref{question:constants} and Question~\ref{question:H-WH}.
\begin{thm}\label{thm:classification}
Let $G$ be a connected simple Lie group. The following are equivalent.
\begin{enumerate}
	\item $G$ is compact or locally isomorphic to $\SO(n,1)$ or $\SU(n,1)$ for some $n\geq 2$.
	\item $G$ has the Haagerup property.
	\item $G$ is weakly amenable with constant 1.
	\item $G$ has the weak Haagerup property with constant 1.
\end{enumerate}
\end{thm}
The purpose of this paper is to consider the same class of groups as in theorem above, but made discrete. When $G$ is a locally compact group, we let $G_\dd$ denote the same group equipped with the discrete topology. The idea of considering Lie groups without their topology (or with the discrete topology, depending on the point of view) is not a new one. For instance, a conjecture of Friedlander and Milnor is concerned with computing the (co)homology of the classifying space of $G_\dd$, when $G$ is a Lie group (see \cite{MR699007} and the survey \cite{MR1059871}).

Other papers discussing the relation between $G$ and $G_\dd$ include \cite{cornulier-jlt}, \cite{MR1181157} and \cite{MR1219692}. Since our focus is approximation properties, will we be concerned with the following question.

\begin{question}\label{question:discrete}
Does the Haagerup property/weak amenability/the weak Haagerup property of $G_\dd$ imply the Haagerup property/weak amenability/the weak Haagerup property of $G$?
\end{question}
It is not reasonable to expect an implication in the other direction. For instance, many compact groups such as $\SO(n)$, $n\geq 3$, are non-amenable as discrete groups. It follows from Theorem~\ref{thm:main} below (see also Proposition~\ref{prop:nonWH}) that when $n\geq 5$, then $\SO(n)$ as a discrete group does not even have the weak Haagerup property. It is easy to see that Question~\ref{question:discrete} has a positive answer for second countable, locally compact groups $G$ that admit a lattice $\Gamma$. Indeed, $G$ has the Haagerup property if and only if $\Gamma$ has the Haagerup property. Moreover, $\Lambda_\WA(\Gamma) = \Lambda_\WA(G)$ and $\Lambda_\WH(\Gamma) = \Lambda_\WH(G)$.
\begin{rem}
A similar question can of course be asked for amenability. This case is already settled: if $G_\dd$ is amenable, then $G$ is amenable \cite[Proposition~4.21]{MR767264}, and the converse is not true in general by the counterexamples mentioned above. A sufficient and necessary condition of the converse implication can be found in \cite{MR1181157}.
\end{rem}
Recall that $\SL(2,\R)$ is locally isomorphic to $\SO(2,1)$ and that $\SL(2,\C)$ is locally isomorphic to $\SO(3,1)$. Thus, Theorem~\ref{thm:classification} and the main theorem below together show in particular that Question~\ref{question:discrete} has a positive answer for connected simple Lie groups. This could however also be deduced (more easily) from the fact that connected simple Lie groups admit lattices \cite[Theorem~14.1]{MR0507234}.
\begin{thm}[Main Theorem]\label{thm:main}
Let $G$ be a connected simple Lie group, and let $G_\dd$ denote the group $G$ equipped with the discrete topology. The following are equivalent.
\begin{enumerate}
	\item $G$ is locally isomorphic to $\SO(3)$, $\SL(2,\R)$, or $\SL(2,\C)$.
	\item $G_\dd$ has the Haagerup property.
	\item $G_\dd$ is weakly amenable with constant 1.
	\item $G_\dd$ is weakly amenable.
	\item $G_\dd$ has the weak Haagerup property with constant 1.
	\item $G_\dd$ has the weak Haagerup property.
\end{enumerate}
\end{thm}

The equivalence of (1) and (2) in Theorem~\ref{thm:main} was already done by de~Cornulier \cite[Theorem~1.14]{cornulier-jlt} and in greater generality. His methods are the inspiration for our proof of Theorem~\ref{thm:main}. That (1) implies (2) basically follows from a theorem of Guentner, Higson and Weinberger \cite[Theorem~5.4]{MR2217050}, namely that the discrete group $\GL(2,K)$ has the Haagerup property for any field $K$. Here we prove a similar statement about weak amenability.

\begin{thm}\label{thm:WA1}
Let $K$ be any field. The discrete group $\GL(2,K)$ is weakly amenable with constant 1.
\end{thm}

Theorem~\ref{thm:WA1} is certainly known to experts. The result was already mentioned in \cite[p.~7]{MR3259044} and in \cite{ozawa-talk} with a reference to \cite{MR2217050}, and indeed our proof of Theorem~\ref{thm:WA1} is merely an adaption of the methods developed in \cite{MR2217050}. However, since no published proof is available, we felt the need to include a proof.

To obtain Theorem~\ref{thm:main} we use the classification of simple Lie groups and then combine Theorem~\ref{thm:WA1} with the following results proved in Section~\ref{sec:failure}:
If $G$ is one of the four groups $\SO(5)$, $\SO_0(1,4)$, $\SU(3)$ or $\SU(1,2)$, then $G_\dd$ does not have the weak Haagerup property.
Also, if $G$ is the universal covering group of $\SU(1,n)$ where $n\geq 2$, then $G_\dd$ does not have the weak Haagerup property.

\section{Preliminaries}\label{sec:approx}

Throughout, $G$ will denote a locally compact group. A kernel $\phi\colon G\times G\to\C$ is a \emph{Schur multiplier} if there exist bounded maps $\xi,\eta\colon G\to\mc H$ into a Hilbert space $\mc H$ such that $\phi(g,h) = \la \xi(g),\eta(h)\ra$ for every $g,h\in G$. The Schur norm of $\phi$ is defined as
$$
\|\phi\|_S = \inf\{ \|\xi\|_\infty\|\eta\|_\infty \}
$$
where the infimum is taken over all $\xi,\eta\colon G\to\mc H$ as above. See \cite[Theorem~5.1]{MR1818047} for different characterizations of Schur multipliers. Clearly, $\|\phi\cdot \psi\|_S\leq \|\phi\|_S\cdot\|\psi\|_S$ and $\|\widecheck\phi\|_S=\|\phi\|_S$ when $\phi$ and $\psi$ are Schur multipliers and $\widecheck\phi(x,y)=\phi(y,x)$. Also, any positive definite kernel $\phi$ on $G$ which is normalized, i.e., $\phi(x,x) = 1$ for every $x\in G$, is a Schur multiplier of norm 1. Finally, notice that the unit ball of Schur multipliers is closed under pointwise limits.

A continuous function $\phi\colon G\to\C$ is a \emph{Herz-Schur multiplier} if the associated kernel $\hat\phi(g,h) = \phi(g^{-1}h)$ is a Schur multiplier. The Herz-Schur norm of $\phi$ is defined as $\|\phi\|_{B_2} = \|\hat\phi\|_S$. When $\phi$ is a Herz-Schur multiplier, the two bounded maps $\xi,\eta\colon G\to\mc H$ can be chosen to be continuous (see \cite{MR753889} and \cite{MR1180643}). The set $B_2(G)$ of Herz-Schur multipliers on $G$ is a unital Banach algebra under pointwise multiplication and $\|\cdot\|_\infty\leq\|\cdot\|_{B_2}$. Any continuous, positive definite function $\phi$ on $G$ is a Herz-Schur multiplier with $\|\phi\|_{B_2} = \phi(1)$.

Below we list a number of permanence results concerning weak amenability and the weak Haagerup property, which will be useful later on. General references containing almost all of the results are \cite{MR1372231}, \cite{MR996553}, \cite{UH-preprint} and \cite{K-WH}. Additionally we refer to \cite[Theorem~III.9]{MR1120720} and \cite[Corollary~12.3.12]{MR2391387}.

Suppose $\Gamma_1$ is a co-amenable subgroup of a discrete group $\Gamma_2$, that is, there exists a left $\Gamma_2$-invariant mean on $l^\infty(\Gamma_2/\Gamma_1)$. Then
\begin{align}\label{eq:co-amenable}
\Lambda_\WA(\Gamma_1)= \Lambda_\WA(\Gamma_2).
\end{align}

If $(G_i)_{i\in I}$ is a directed family of open subgroups in a locally compact group $G$ whose union is $G$, then
\begin{align}\label{eq:union}
\Lambda_\WA(G)=\sup\Lambda_\WA(G_i).
\end{align}

For any two locally compact groups $G$ and $H$
\begin{align}\label{eq:product}
\Lambda_\WA(G\times H)=\Lambda_\WA(G)\Lambda_\WA(H).
\end{align}

When $H$ is a closed subgroup of $G$
\begin{align}\label{eq:subgroup}
\Lambda_\WA(H) \leq \Lambda_\WA(G)
\quad\text{and}\quad
\Lambda_\WH(H) \leq \Lambda_\WH(G).
\end{align}

When $K$ is a compact normal subgroup of $G$ then
\begin{align}\label{eq:mod-compact}
\Lambda_\WA(G/K) = \Lambda_\WA(G)
\quad\text{and}\quad
\Lambda_\WH(G/K) = \Lambda_\WH(G).
\end{align}

When $Z$ is a central subgroup of a discrete group $G$ then
\begin{align}\label{eq:mod-central}
\Lambda_\WA(G) \leq \Lambda_\WA(G/Z).
\end{align}

Recall that a \emph{lattice} in a locally compact group $G$ is a discrete subgroup $\Gamma$ such that the quotient $G/\Gamma$ admits a non-trivial finite $G$-invariant Radon measure. When $\Gamma$ is a lattice in a second countable, locally compact $G$ then
\begin{align}\label{eq:lattice}
\Lambda_\WA(\Gamma) = \Lambda_\WA(G)
\quad\text{and}\quad
\Lambda_\WH(\Gamma) = \Lambda_\WH(G).
\end{align}

When $H$ is a finite index, closed subgroup in a group $G$ then
\begin{align}\label{eq:finite-index}
\Lambda_\WH(H) = \Lambda_\WH(G).
\end{align}

\section{Weak amenability of \texorpdfstring{$\GL(2,K)$}{GL(2,K)}}\label{sec:WA}
This section is devoted to the proof of Theorem~\ref{thm:WA1} (see Theorem~\ref{thm:GL} below). The general idea of our proof follows the idea of \cite[Section~5]{MR2217050}, where it is shown that for any field $K$ the discrete group $\GL(2,K)$ has the Haagerup property. Our proof of Theorem~\ref{thm:WA1} also follows the same strategy as used in \cite{MR2764895}.

Recall that a \emph{pseudo-length function} on a group $G$ is a function $\ell\colon G\to [0,\infty)$ such that
\begin{itemize}
	\item $\ell(e)=0$,
	\item $\ell(g)=\ell(g^{-1})$,
	\item $\ell(g_1g_2)\leq \ell(g_1)+\ell(g_2)$.
\end{itemize}
Moreover, $\ell$ is a length function on $G$ if, in addition, $\ell(g)=0 \implies g=e$.

\begin{defi}
We say that the pseudo-length group $(G,\ell)$ is \emph{weakly amenable} if there exist a sequence $(\phi_n)$ of Herz-Schur multipliers on $G$ and a sequence $(R_n)$ of positive numbers such that
\begin{itemize}
\item $\sup_n \|\phi_n\|_{B_2}<\infty$;
\item $\supp \phi_n \subseteq \{g\in G\mid \ell(g)\leq R_n\}$;
\item $\phi_n\to 1$ uniformly on $\{g\in G\mid \ell(g)\leq S\}$ for every $S>0$.
\end{itemize}
The \emph{weak amenability constant} $\Lambda_\WA(G,\ell)$  is defined as the best possible constant $\Lambda$ such that $\sup_n \|\phi_n\|_{B_2} \leq\Lambda$, where $(\phi_n)$ is as just described.
\end{defi}
Notice that if the group $G$ is discrete and the pseudo-length function $l$ on $G$ is proper (in particular, $G$ is countable), then the weak amenability of $(G,l)$ is equivalent to the weak amenability of $G$ with same weak amenability constant. On other hand, every countable discrete group admits a proper length function, which is unique up to coarse equivalence (\cite[Lemma~2.1]{MR1871980}). If the group is finitely generated discrete, one can simply take the word-length function associated to any finite set of generators.

The next proposition is a variant of a well-known theorem, which follows from two classical results:
\begin{itemize}
	\item The graph distance $\dist$ on a tree $T$ is a conditionally negative definite kernel \cite{MR520930}.
	\item The Schur multiplier associated with the characteristic function $\chi_n$ of the subset $\{(x,y)\in T^2\mid \dist (x,y)=n\}$ has Schur norm at most $2n$ for every $n\in\N$ \cite[Proposition~2.1]{MR1119263}.
\end{itemize}
The proof below is similar to the proof of {\cite[Corollary~12.3.5]{MR2391387}}.
\begin{prop}\label{prop:tree}
Suppose a group $G$ acts isometrically on a tree $T$ and that $\ell$ is a pseudo-length function on $G$. Suppose moreover $\dist(g.v,v)\to \infty$ if and only if $\ell(g)\to \infty$ for some (and hence every) vertex $v\in T$. Then $\Lambda_\WA(G,\ell)=1$.
\end{prop}
\begin{proof}
Fix a vertex $v\in T$ as in the assumptions. For every $n\in \N$ we consider the functions $\psi_n(g)=\exp(-\frac{1}{n}\dist(g.v,v))$ and $\dot{\chi}_n(g)=\chi_n(g.v,v)$ defined for $g\in G$. Then
$$
\dot{\chi}_m(g)\psi_n(g)=\exp(-m/n)\dot{\chi}_m(g)
$$
holds for all $g\in G$ and every $n,m\in \N$. As $G$ acts isometrically on $T$, each $\psi_n$ is a unital positive definite function on $G$ by Schoenberg's theorem and $\|\dot{\chi}_n\|_{B_2}\leq 2n$ for every $n\in\N$. It follows that $\|\psi_n\|_{B_2}=1$ and $\|\dot{\chi}_m\psi_n\|_{B_2}\leq 2m\cdot\exp(-m/n)$ for every $n,m\in\N$. Therefore, for any $M\in\N$, we have
\begin{align*}
\left\|\sum_{m=0}^M\dot{\chi}_m\psi_n\right\|_{B_2}\leq \|\psi_n\|_{B_2}+\left\|\sum_{m>M}\dot{\chi}_m\psi_n\right\|_{B_2}\leq 1+\sum_{m>M}2m\cdot\exp(-m/n).
\end{align*}
Hence, if we choose $M_n$ suitably for all $n\in \N$, then the functions $\phi_n=\sum_{m=0}^{M_n}\dot{\chi}_m\psi_n$ satisfy that $\|\phi_n\|_{B_2}\leq 1+\frac1n$ and $\supp\phi_n\subseteq \{g\in G\mid \dist(g.v,v)\leq M_n\}$. The assumption
$$
\dist(g.v,v)\to\infty\iff\ell(g)\to\infty
$$
then insures that $\supp\phi_n\subseteq\{g\in G\mid \ell(g)\leq R_n\}$ for some suitable $R_n$ and that $\phi_n\to 1$ uniformly on $\{g\in G\mid \ell(g)\leq S\}$ for every $S>0$, as desired.
\end{proof}

\begin{rem}
The two classical results listed above have a generalization:
\begin{itemize}
    \item The combinatorial distance $\dist$ on the 1-skeleton of a CAT(0) cube complex $X$ is a conditionally negative definite kernel on the vertex set of $X$ \cite{MR1432323}.
	\item The Schur multiplier associated with the characteristic function of the subset $\{(x,y)\in X^2\mid \dist (x,y)=n\}$ has Schur norm at most $p(n)$ for every $n\in\N$, where $p$ is a polynomial and $X$ is (the vertex set of) a finite-dimensional CAT(0) cube complex \cite[Theorem~2]{MR2381161}.
\end{itemize}
To see that these results are in fact generalizations, we only have to notice that a tree is exactly a one-dimensional CAT(0) cube complex, and in this case the combinatorial distance is just the graph distance. Because of these generalizations and the fact that the exponential function increases faster than any polynomial, it follows with the same proof as the proof of Proposition \ref{prop:tree} that the following generalization is true (see also {\cite[Theorem~3]{MR2381161}}): suppose a group $G$ acts cellularly (and hence isometrically) on a finite-dimensional CAT(0) cube complex $X$ and that $\ell$ is a pseudo-length function on $G$. Suppose moreover $\dist(g.v,v)\to \infty$ if and only if $\ell(g)\to \infty$ for some (and hence every) vertex $v\in X$. Then $\Lambda_\WA(G,\ell)=1$.
\end{rem}

In our context, a \emph{norm} on a field $K$ is a map $d\colon K\to [0,\infty)$ satisfying, for all $x,y\in K$
\begin{enumerate}
	\item[(i)] $d(x)=0$ implies $x=0$,
	\item[(ii)] $d(xy)=d(x)d(y)$,
	\item[(iii)] $d(x+y)\leq d(x)+d(y)$.
\end{enumerate}
A norm obtained as the restriction of the usual absolute value on $\C$ via a field embedding $K\hookrightarrow\C$ is \emph{archimedean}. A norm is \emph{discrete} if the triangle inequality (iii) can be replaced by the stronger ultrametric inequality
\begin{enumerate}
	\item[(iii')] $d(x+y)\leq \max\{d(x),d(y)\}$
\end{enumerate}
and the range of $d$ on $K^{\times}$ is a discrete subgroup of the multiplicative group $(0,\infty)$.

\begin{thm}[{\cite[Theorem~2.1]{MR2217050}}]\label{thm:DE}
Every finitely generated field $K$ is \emph{discretely embeddable:} For every finitely generated subring $A$ of $K$ there exists a sequence of norms $d_n$ on $K$, each either archimedean or discrete, such that for every sequence $R_n>0$, the subset
$$
\{a\in A \mid d_n(a)\leq R_n \text{ for all } n\in \N\}
$$
is finite.
\end{thm}

Let $d$ be a norm on a field $K$. Following Guentner, Higson and Weinberger \cite{MR2217050} define a pseudo-length function $\ell_d$ on $\GL(n,K)$ as follows: if $d$ is discrete
\begin{align*}
\ell_d(g)=\log\max_{i,j}\{d(g_{ij}),d(g^{ij})\},
\end{align*}
where $g_{ij}$ and $g^{ij}$ are the matrix coefficients of $g$ and $g^{-1}$, respectively; if $d$ is archimedean, coming from an embedding of $K$ into $\C$ then
\begin{align*}
\ell_d(g)=\log\max\{\|g\|,\|g^{-1}\|\},
\end{align*}
where $\|\cdot\|$ is the operator norm of a matrix in $\GL(n,\C)$.
\begin{prop}\label{prop:AD}
Let $d$ be an archimedean or a discrete norm on a field $K$. Then the pseudo-length group $(\SL(2,K),\ell_d)$ is weakly amenable with constant $1$.
\end{prop}
\begin{proof}
The archimedean case: it is clear that the pseudo-length function on $\SL(2,K)$ is the restriction of that on $\SL(2,\C)$, so clearly we only have to show $(\SL(2,\C),\ell_d)$ is weakly amenable with constant $1$. Since $\ell_d$ is continuous and proper, this follows from the fact that $\SL(2,\C)$ is weakly amenable with constant $1$ as a locally compact group (\cite[Remark 3.8]{MR784292}).

The discrete case: this is a direct application of \cite[Lemma~5.9]{MR2217050} and Proposition~\ref{prop:tree}. Indeed, \cite[Lemma~5.9]{MR2217050} states that there exist a tree $T$ and a vertex $v_0\in T$ such that $\SL(2,K)$ acts isometrically on $T$ and
$$
\dist(g.v_0,v_0)=2\max_{i,j}-\frac{\log d(g_{ij})}{\log d(\pi)},
$$
for all $g=[g_{ij}]\in \SL(2,K)$. Here $\dist$ is the graph distance on $T$ and $\pi$, the uniformizer, is certain element of $\{x\in K\mid d(x)<1\}$. Since the action is isometric, $\dist(g.v_0,v_0)\to \infty$ if and only if $\ell_d(g)\to \infty$. Hence, we are done by Proposition~\ref{prop:tree}.
\end{proof}
\begin{cor}
Let $K$ be a field and $G$ a finitely generated subgroup of $\SL(2,K)$. Then there exists a sequence of pseudo-length functions $\ell_n$ on $G$ such that $\Lambda_\WA(G,\ell_n) = 1$ for every $n$, and such that for any sequence $R_n>0$, the set $\bigcap_n\{g\in G\mid \ell_n(g)\leq R_n\}$ is finite.
\end{cor}
\begin{proof}
As $G$ is finitely generated, we may assume that $K$ is finitely generated as well. Now, let $A$ be the finitely generated subring of $K$ generated by the matrix coefficients of a finite generating set for $G$. Clearly, $G\subseteq\SL(2,A)\subseteq\SL(2,K)$. Since $K$ is discretely embeddable, we may choose a sequence of norms $d_n$ on $K$ according to Theorem~\ref{thm:DE}. It follows from Proposition~\ref{prop:AD} that $\Lambda_\WA(G,\ell_{d_n})=1$. We complete the proof by observing that for any sequence $R_n>0$,
\begin{align*}
\bigcap_n\{g\in G\mid \ell_{d_n}(g)\leq R_n\}\subseteq \SL(2,F),
\end{align*}
where $F$ is the finite set $\{a\in A \mid d_n(a)\leq \exp(R_n) \text{ for all } n\in \N\}$.
\end{proof}
\begin{thm}\label{thm:GL}
Let $K$ be a field. Every subgroup $\Gamma$ of $\GL(2,K)$ is weakly amenable with constant $1$ (as a discrete group).
\end{thm}
\begin{proof}
By the permanence results listed in Section~\ref{sec:approx} we can reduce our proof to the case where $\Gamma$ is a finitely generated subgroup of $\SL(2,K)$. It then follows from the previous corollary that there exists a sequence $\ell_n$ of pseudo-length functions on $\Gamma$ such that $\Lambda_\WA(\Gamma,\ell_n)=1$ and for any sequence $R_n>0$, the set $\bigcap_n\{g\in \Gamma\mid \ell_n(g)\leq R_n\}$ is finite.

For each fixed $n\in \N$ there is a sequence ($\phi_{n,k})_k$ of Herz-Schur multipliers on $\Gamma$ and a sequence of positive numbers $(R_{n,k})_k$ such that
\begin{enumerate}
	\item $\|\phi_{n,k}\|_{B_2}\leq 1$ for all $k\in \N$;
	\item $\supp\phi_{n,k}\subseteq \{g\in \Gamma\mid \ell_n(g)\leq R_{n,k}\}$;
	\item\label{item:conv} $\phi_{n,k}\to 1$ uniformly on $\{g\in \Gamma\mid \ell_n(g)\leq S\}$ for every $S>0$ as $k\to \infty$.
\end{enumerate}
Upon replacing $\phi_{n,k}$ by $|\phi_{n,k}|^2$ we may further assume that $0\leq \phi_{n,k}\leq 1$ for all $n,k\in \N$.

Given any $\epsilon>0$ and any finite subset $F\subseteq \Gamma$, we choose a sequence $0<\epsilon_n < 1$ such that $\prod_n(1-\epsilon_n)>1-\epsilon$. It follows from \eqref{item:conv} that for each $n\in \N$ there exists $k_n\in \N$ such that $1-\epsilon_n<\phi_{n,k_n}(g)$ for all $g\in F$. Consider the function $\phi=\prod_n\phi_{n,k_n}$. It is not hard to see that $\phi$ is well-defined, since $0\leq \phi_{n,k_n}\leq 1$. Additionally, since $\|\phi_{n,k_n}\|_{B_2}\leq 1$ for all $n\in \N$ we also have $\|\phi\|_{B_2}\leq 1$. Moreover, $\supp\phi\subseteq \bigcap_n\{g\in \Gamma\mid \ell_n(g)\leq R_{n,k_n}\}$ and
$$
\phi(g)=\prod_n\phi_{n,k_n}(g)>\prod_n(1-\epsilon_n)>1-\epsilon
$$
for all $g\in F$. This completes the proof.
\end{proof}
The remaining part of this section follows de~Cornulier's idea from \cite{cornulier-bbms}. In \cite{cornulier-bbms} he proved the same results for Haagerup property, and the same argument actually works for weak amenability with constant $1$.
\begin{cor}
Let $R$ be a unital commutative ring without nilpotent elements. Then every subgroup $\Gamma$ of $\GL(2,R)$ is weakly amenable with constant $1$ (as a discrete group).
\end{cor}
\begin{proof}
Again by the permanence results in Section~\ref{sec:approx}, we may assume that $\Gamma$ is a finitely generated subgroup of $\SL(2,R)$, and hence that $R$ is also finitely generated. It is well-known that every finitely generated ring is Noetherian and in such a ring there are only finitely many minimal prime ideals. Let $\mf p_1,\ldots,\mf p_n$ be the minimal prime ideals in $R$. The intersection of all minimal prime ideals is the set of nilpotent elements in $R$, which is trivial by our assumption. So $R$ embeds into the finite product $\prod_{i=1}^n R/\mf p_i$. If $K_i$ denotes the fraction field of the integral domain $R/\mf p_i$, then $\Gamma$ embeds into $\SL(2,\prod_{i=1}^n K_i)=\prod_{i=1}^n\SL(2,K_i)$. Now, the result is a direct consequence of Theorem~\ref{thm:GL}, \eqref{eq:product} and \eqref{eq:subgroup}.
\end{proof}
\begin{rem}
In the previous corollary and also in Theorem~\ref{thm:GL}, the assumption about commutativity cannot be dropped. Indeed, the group $\SL(2,\mathbb{H})$ with the discrete topology is not weakly amenable, where $\mathbb{H}$ is the skew-field of quaternions. This can be seen from Theorem~\ref{thm:main}. Moreover, $\SL(2,\mathbb H)_\dd$ does not even have the weak Haagerup property by the same argument.
\end{rem}
\begin{rem}
In the previous corollary, the assumption about the triviality of the nilradical cannot be dropped. Indeed, we show now that the group $\SL(2,\Z[x]/x^2)$ is not weakly amenable. The essential part of the argument is Dorofaeff's result that the locally compact group $\R^3\rtimes\SL(2,\R)$ is not weakly amenable \cite{MR1245415}. Here the action $\SL(2,\R)\acts\R^3$ is the unique irreducible 3-dimensional representation of $\SL(2,\R)$.

Consider the ring $R = \R[x]/x^2$. We write elements of $R$ as polynomials $ax + b$ where $a,b\in\R$ and $x^2 = 0$. Consider the unital ring homomorphism $\phi\colon R\to\R$ given by setting $x = 0$, that is, $\phi(ax+b) = b$. Then $\phi$ induces a group homomorphism $\tilde\phi\colon\SL(2,R)\to\SL(2,\R)$. Embedding $\R\subseteq R$ as constant polynomials, we obtain an embedding $\SL(2,\R)\subseteq\SL(2,R)$ showing that $\tilde\phi$ splits. The kernel of $\tilde\phi$ is easily identified as
$$
\ker\tilde\phi = \left\{ \begin{pmatrix}
	a_{11}x+1 & a_{12}x \\
	a_{21}x & a_{22}x+1
\end{pmatrix} \middle\vert
a_{ij}\in\R,\ a_{11}+a_{22} = 0
\right\}
\simeq \mf{sl}(2,\R)
$$
We deduce that $\SL(2,R)$ is the semidirect product $\mf{sl}(2,\R)\rtimes\SL(2,\R)$. A simple computation shows that the action $\SL(2,\R)\acts\mf{sl}(2,\R)$ is the adjoint action. Since $\mf{sl}(2,\R)$ is a simple Lie algebra, the adjoint action is irreducible. By uniqueness of the 3-dimensional irreducible representation of $\SL(2,\R)$ (see \cite[p.~107]{MR0430163}) and from \cite{MR1245415} we deduce that $\mf{sl}(2,\R)\rtimes\SL(2,\R)\simeq\R^3\rtimes\SL(2,\R)$ is not weakly amenable.

It is easy to see that $\SL(2,\Z[x]/x^2)$ is identified with $\mf{sl}(2,\Z)\rtimes\SL(2,\Z)$ under the isomorphism $\SL(2,R)\simeq\mf{sl}(2,\R)\rtimes\SL(2,\R)$. Since $\mf{sl}(2,\Z)\rtimes\SL(2,\Z)$ is a lattice in $\mf{sl}(2,\R)\rtimes\SL(2,\R)$, we conclude from \eqref{eq:lattice} that $\mf{sl}(2,\Z)\rtimes\SL(2,\Z)$ and hence $\SL(2,\Z[x]/x^2)$ is not weakly amenable.
\end{rem}

\begin{rem}
We do not know if $\SL(2,\Z[x]/x^2)$ also fails to have the weak Haagerup property. As $\SL(2,\Z[x]/x^2)$ may be identified with a lattice in $\R^3\rtimes\SL(2,\R)$, by \eqref{eq:lattice} the question is equivalent to the question \cite[Remark~5.3]{HK-WH-examples} raised by Haagerup and the first author concerning the weak Haagerup property of the group $\R^3\rtimes\SL(2,\R)$.
\end{rem}

Recall that a group $\Gamma$ is residually free if for every $g\neq 1$ in $\Gamma$, there is a homomorphism $f$ from $\Gamma$ to a free group $F$ such that $f(g)\neq 1$ in $F$. Equivalently, $\Gamma$ embeds into a product of free groups of rank two. A group $\Gamma$ is residually finite if for every $g\neq 1$ in $\Gamma$, there is a homomorphism $f$ from $\Gamma$ to a finite group $F$ such that $f(g)\neq 1$ in $F$. Equivalently, $\Gamma$ embeds into a product of finite groups. Since free groups are residually finite, it is clear that residually free groups are residually finite. On the other hand, residually finite groups need not be residually free as is easily seen by considering e.g. groups with torsion.
\begin{cor}
Every residually free group is weakly amenable with constant $1$.
\end{cor}
\begin{proof}
Since the free group of rank two can be embedded in $\SL(2,\Z)$, a residually free group embeds in $\prod_{i\in I}\SL(2,\Z)=\SL(2,\prod_{i\in I}\Z)$ for a suitably large set $I$. We complete the proof by the previous corollary. 
\end{proof}

\section{Failure of the weak Haagerup property}\label{sec:failure}
In this section we will prove the following result.
\begin{prop}\label{prop:nonWH}
If $S$ is one of the four groups $\SO(5)$, $\SO_0(1,4)$, $\SU(3)$ or $\SU(1,2)$, then $S_\dd$ does not have the weak Haagerup property.

Also, if $S$ is the universal covering group of $\SU(1,n)$ where $n\geq 2$, then $S_\dd$ does not have the weak Haagerup property.
\end{prop}

When $p,q\geq 0$ are integers, not both zero, and $n = p+q$, we let $I_{p,q}$ denote the diagonal $n\times n$ matrix with $1$ in the first $p$ diagonal entries and $-1$ in the last $q$ diagonal entries. When $g$ is a complex matrix, $g^t$ denotes the transpose of $g$, and $g^*$ denotes the adjoint (conjugate transpose) of $g$. We recall that
\begin{align*}
\SO(p,q) &= \{g\in\SL(p+q,\R) \mid g^t I_{p,q} g = I_{p,q} \} \\
\SO(p,q,\C) &= \{g\in\SL(p+q,\C) \mid g^t I_{p,q} g = I_{p,q} \} \\
\SU(p,q) &= \{g\in\SL(p+q,\C) \mid g^* I_{p,q} g = I_{p,q} \}.
\end{align*}
When $p,q > 0$, the group $\SO(p,q)$ has two connected components, and $\SO_0(p,q)$ denotes the identity component. In particular, by \eqref{eq:finite-index}, the group $\SO(p,q)_\dd$ has the weak Haagerup property if and only if the group $\SO_0(p,q)_\dd$ has the weak Haagerup property.

\begin{proof}[Proof of Proposition~\ref{prop:nonWH}] We follow a strategy that we have learned from de~Cornulier \cite{cornulier-jlt}, where the same techniques are applied in connection with the Haagerup property. The idea of the proof is the following.

If $Z$ denotes the center of $S$, then we consider the group $S/Z$ as a real algebraic group $G(\R)$ with complexification $G(\C)$. Let $K$ be a number field of degree three over $\Q$, not totally real, and let $\mc O$ be its ring of integers. Then by the Borel Harish--Chandra Theorem (see \cite[Theorem~12.3]{MR0147566} or \cite[Proposition~5.42]{morris}), $G(\mc O)$ embeds diagonally as a lattice in $G(\R) \times G(\C)$. If $\Gamma$ is the inverse image in $S\times G(\C)$ of $G(\mc O)$, then $\Gamma$ is a lattice in $S\times G(\C)$.

The group $G(\C)$ has real rank at least two, and we deduce that $\Gamma$ does not have the weak Haagerup property by combining \cite[Theorem~B]{HK-WH-examples} with \eqref{eq:lattice}. The projection $S\times G(\C)\to S$ is injective on $\Gamma$, and hence \eqref{eq:subgroup} implies that $S_\dd$ also does not have the weak Haagerup property.
\end{proof}

\section{Proof of the Main Theorem}
In this section we prove Theorem~\ref{thm:main}. The theorem is basically a consequence of Theorem~\ref{thm:WA1} and Proposition~\ref{prop:nonWH} together with the permanence results listed in Section~\ref{sec:approx} and general structure theory of simple Lie groups.

When two Lie groups $G$ and $H$ are locally isomorphic we write $G\approx H$. An important fact about Lie groups and local isomorphims is the following \cite[Theorem~II.1.11]{MR514561}: Two Lie groups are locally isomorphic if and only if their Lie algebras are isomorphic.

The following is extracted from \cite[Chapter~II]{MR0015396} and \cite[Section~I.11]{MR1920389} to which we refer for details. If $G$ is a connected Lie group, there exists a connected, simply connected Lie group $\tilde G$ and a covering homomorphism $\tilde G\to G$. The kernel of the covering homomorphism is a discrete, central subgroup of $\tilde G$, and it is isomorphic to the fundamental group of $G$. The group $\tilde G$ is called the \emph{universal covering group} of $G$. Clearly, $\tilde G$ and $G$ are locally isomorphic. Conversely, any connected Lie group locally isomorphic to $G$ is the quotient of $\tilde G$ by a discrete, central subgroup. If $N$ is a discrete subgroup of the center $Z(\tilde G)$ of $\tilde G$, then the center of $\tilde G/N$ is $Z(\tilde G)/N$.

Let $G_1$ and $G_2$ be locally compact groups. We say that $G_1$ and $G_2$ are \emph{strongly locally isomorphic}, if there exist a locally compact group $G$ and finite normal subgroups $N_1$ and $N_2$ of $G$ such that $G_1\simeq G/N_1$ and $G_2\simeq G/N_2$. In this case we write $G_1\sim G_2$. It follows from \eqref{eq:mod-compact} that if $G\sim H$, then $\Lambda_\WH(G_\dd) = \Lambda_\WH(H_\dd)$.

A theorem due to Weyl states that a connected, simple, compact Lie group has a compact universal cover with finite center \cite[Theorem~12.1.17]{MR3025417}, \cite[Theorem~II.6.9]{MR514561}. Thus, for connected, simple, compact Lie groups $G$ and $H$, $G\approx H$ implies $G\sim H$.

\begin{proof}[Proof of Theorem~\ref{thm:main}]
Let $G$ be a connected simple Lie group. As mentioned, the equivalence (1)$\iff$(2) was already done by de~Cornulier \cite[Theorem~1.14]{cornulier-jlt} in a much more general setting, so we leave out the proof of this part. We only prove the two implications (1)$\implies$(3) and (6)$\implies$(1), since the remaining implications then follow trivially.

Suppose (1) holds, that is, $G$ is locally isomorphic to $\SO(3)$, $\SL(2,\R)$ or $\SL(2,\C)$. If $Z$ denotes the center of $G$, then by assumption $G/Z$ is isomorphic to $\SO(3)$, $\PSL(2,\R)$ or $\PSL(2,\C)$. It follows from Theorem~\ref{thm:WA1} and \eqref{eq:mod-compact} that the groups $\SO(3)$, $\PSL(2,\R)$ and $\PSL(2,\C)$ equipped with the discrete topology are weakly amenable with constant 1 (recall that $\SO(3)$ is a subgroup of $\PSL(2,\C)$). From \eqref{eq:mod-central} we deduce that $G_\dd$ is weakly amenable with constant 1. This proves (3).

Suppose (1) does not hold. We prove that (6) fails, that is, $G_\dd$ does not have the weak Haagerup property. We divide the proof into several cases depending on the real rank of $G$. We recall that with the Iwasawa decomposition $G = KAN$, the real rank of $G$ is the dimension of the abelian group $A$.

If the real rank of $G$ is at least two, then $G$ does not have the weak Haagerup property \cite[Theorem~B]{HK-WH-examples}. By a theorem of Borel, $G$ contains a lattice (see \cite[Theorem~14.1]{MR0507234}), and by \eqref{eq:lattice} the lattice also does not have the weak Haagerup property. We conclude that $G_\dd$ does not have the weak Haagerup property.

If the real rank of $G$ equals one, then the Lie algebra of $G$ is isomorphic to a Lie algebra in the list \cite[(6.109)]{MR1920389}. See also \cite[Ch.X \S6]{MR514561}. In other words, $G$ is locally isomorphic to one of the classical groups $\SO_0(1,n)$, $\SU(1,n)$, $\Sp(1,n)$ for some $n\geq 2$ or locally isomorphic to the exceptional group $\FF$. Here $\SO_0(1,n)$ denotes the identity component of the group $\SO(1,n)$.

We claim that the universal covering groups of $\SO_0(1,n)$, $\Sp(1,n)$ and $\FF$ have finite center except for the group $\SO_0(1,2)$. Indeed, $\Sp(1,n)$ and $\FF$ are already simply connected with finite center. The $K$-group from the Iwasawa decomposition of $\SO_0(1,n)$ is $\SO(n)$ which has fundamental group of order two, except when $n = 2$, and hence $\SO_0(1,n)$ has fundamental group of order two as well. As the center of the universal cover is an extension of the center of $\SO_0(1,n)$ by the fundamental group of $\SO_0(1,n)$, the claim follows.

The universal covering group $\tilde\SU(1,n)$ of $\SU(1,n)$ has infinite center isomorphic to the group of integers.

We have assumed that $G$ is not locally isomorphic to $\SL(2,\R)\sim\SO_0(1,2)$ or $\SL(2,\C)\sim\SO_0(1,3)$. If $G$ has finite center, it follows that $G$ is strongly locally isomorphic to one of the groups
\begin{align*}
\begin{array}{ll}
\SO_0(1,n),& n\geq 4, \\
\SU(1,n),& n\geq 2, \\
\Sp(1,n),& n\geq 2, \\
\FF,
\end{array}
\end{align*}
and if $G$ has infinite center, then $G$ is isomorphic to $\tilde\SU(1,n)$. Clearly, there are inclusions
\begin{align*}
\SO_0(1,4)\subseteq\SO_0(1,n),\ n\geq 4,\\
\SU(1,2)\subseteq\SU(1,n),\ n\geq 2,\\
\SU(1,2)\subseteq\Sp(1,n),\ n\geq 2.
\end{align*}
The cases where $G$ is strongly locally isomorphic to $\SO_0(1,n)$, $\SU(1,n)$ or $\Sp(1,n)$ are then covered by Proposition~\ref{prop:nonWH}. Since $\SO(5)\subseteq\SO(9)\sim\Spin(9)\subseteq\FF$ (\cite[\S.4.Proposition~1]{MR560851}), the case where $G\sim\FF$ is also covered by Proposition~\ref{prop:nonWH}. Finally, if $G\simeq\tilde\SU(1,n)$, then Proposition~\ref{prop:nonWH} shows that $G_\dd$ does not have weak Haagerup property.

If the real rank of $G$ is zero, then it is a fairly easy consequence of \cite[Theorem~12.1.17]{MR3025417} that $G$ is compact. Moreover, the universal covering group of $G$ is compact and with finite center.

By the classification of compact simple Lie groups as in Table~IV of \cite[Ch.X \S6]{MR514561} we know that $G$ is strongly locally isomorphic to one of the groups $\SU(n+1)$ ($n\geq 1$), $\SO(2n + 1)$ ($n\geq 2$), $\Sp(n)$ ($n\geq 3$), $\SO(2n)$ ($n\geq 4$) or one of the five exceptional groups
$$
E_6,\ E_7,\ E_8,\ F_4,\ G_2.
$$
By assumption $G$ is not strongly locally isomorphic to $\SU(2)\sim\SO(3)$. Using \eqref{eq:mod-compact} it then suffices to show that if $G$ equals any other group in the list, then $G_\dd$ does not have the weak Haagerup property. Clearly, there are inclusions
\begin{align*}
\SO(5)\subseteq\SO(n),\ n\geq 5,\\
\SU(3)\subseteq\SU(n),\ n\geq 3,\\
\SU(3)\subseteq\Sp(n),\ n\geq 3.
\end{align*}
Since we also have the following inclusions among Lie algebras (Table~V of \cite[Ch.X \S6]{MR514561})
$$
\mf{so}(5)\subseteq\mf{so}(9)\subseteq\mf f_4\subseteq\mf e_6\subseteq\mf e_7\subseteq\mf e_8
$$
and the inclusion (\cite{MR571796})
$$
\SU(3)\subseteq G_2,
$$
it is enough to consider the cases where $G = \SO(5)$ or $G = \SU(3)$. These two cases are covered by Proposition~\ref{prop:nonWH}. Hence we have argued that also in the real rank zero case $G_\dd$ does not have the weak Haagerup property.
\end{proof}

\section*{Acknowledgements}
The authors wish to thank U.~Haagerup for helpful discussions on the subject.


\end{document}